\documentclass[12 pt]{article}

\usepackage{setspace,amsthm,amssymb,graphicx,thesiscommands,enumerate}
\usepackage[T1]{fontenc}

\title{Lifting Arc Diagrams Under Branched Covers: An Inverse Problem and its Solution}
\author{Cyrus A. Peterpaul}

\begin{document}
\maketitle

\begin{abstract}
    A branched covering map of surfaces induces a map in the opposite direction between their arc complexes. We represent a branched covering map combinatorially using what we call a lifting picture, and use this representation to computably solve the membership problem of the set of weighted arc diagrams on a given surface which can be obtained by lifting a weighted arc diagram on a bigon. We provide a brute force solution in the general case, and an efficient solution when the input arc diagram is a triangulation.
\end{abstract}

\section{Introduction}
Arc diagrams are simple, combinatorial objects associated to surfaces with boundary. They consist of homotopy classes of disjoint curves, and can be thought of as embedded graphs on suitably marked surfaces. Arc diagrams and the simplicial complexes, known as arc complexes, which can be built from them, are have been studied by topologists for many years; see, for example, \cite{hatcher1991triangulations}, \cite{korkmaz_papadopoulos_2010}, \cite{fomin2008cluster}. This paper examines the behavior of weighted arc diagrams (that is, diagrams with nonnegative real numbers assigned to each arc) under topological branched covering maps. A given branched covering map of marked surfaces induces a map between arcs of the base space and arcs of the total space by path lifting.  

We are interested in the membership problem of the set of weighted arc diagrams which can be realized by such a process of lifting from a disk with two marked boundary points, which is referred to as a bigon. That is, given a weighted arc diagram, decide if it can be realized by lifting the arcs of a diagram on a bigon under a suitable branched cover. In section \ref{realizability}, we construct a combinatorial way of representing a topological branched cover which we refer to as a lifting picture. Because there are finitely many lifting pictures on a given surface, this problem admits a finite, if computationally expensive, solution, which we present as algorithm \ref{brute-force}. 

In the sections which follow, we present a much more efficient solution to this problem in the case of an arc diagram which is a triangulation of the surface on which it lives. We tackle this by first presenting solutions to two special cases in sections \ref{homonymous recursion section} and \ref{heteronymous section}, then merging these two special case solutions to handle a general triangulation in section \ref{general maximal section}.

\section{Definitions and Notation}
We first need to define a particular kind of marked surface with boundary, which we will call a \emph{substrate}. The objects and definitions here broadly track with those in \cite{fomin2008cluster}.
\begin{define}
Let $\Sigma$ be a surface with boundary. Let $V$ be a finite set of distinct points on the boundary of $\Sigma$ which we call \emph{vertices}. A \emph{2-coloring} is a map $V \to \{\alpha,\beta\}$ so that no two neighboring vertices are assigned the same value.
\end{define}

\begin{define}
A \emph{substrate} is a surface $\Sigma$ with boundary equipped with a set $M_\Sigma$ of marked points in the interior of $\Sigma$, a set of vertices $V_\Sigma$, and a 2-coloring $C_\Sigma:V_\sigma \to \{\alpha,\beta\}$. Each boundary component of $\Sigma$ must contain at least 2 vertices. Call the componenets of $\bdry \Sigma \backslash V_\Sigma$ \emph{boundary edges}.
\end{define}

The set $M_\Sigma$ will typically be empty except for bigons, which will be defined shortly. Substrates are the homes of arc diagrams. 

\begin{define}
Let $\Sigma$ be a substrate and $p:[0,1] \to \Sigma$ be a simple path with $p(0),p(1) \in V_\Sigma$, and $p(t) \not\in M_\Sigma$ for all $t$. An \emph{arc} is the homotopy class relative to $\{0,1\}$ of such a path through paths disjoint from $M_\Sigma$.
\end{define}

We will simply say \emph{homotopic} to mean \emph{homotopic rel endpoints through paths disjoint from} $M_\Sigma$ in the context of arcs. An arc is called \emph{trivial} if it is homotopic to a path in the boundary of $\Sigma$ which intersect $V_\Sigma$ only at its endpoints. Otherwise, it is called \emph{nontrivial}. In particular, an arc which is homotopic to a vertex is trivial. An arc is called \emph{homonymous} if $C_\Sigma$ takes the same value at both of its endpoints; that is, if its endpoints are the same color. Otherwise, an arc is called \emph{heteronymous}.

\begin{define}
Two arcs $a$ and $b$ are \emph{noncrossing} if there are paths representing $a$ and $b$ that are disjoint except at their endpoints. Otherwise, $a$ and $b$ \emph{cross}.
\end{define}

\begin{define}
An \emph{arc diagram} on a substrate $\Sigma$ is a set of mutually noncrossing arcs on $\Sigma$. 
\end{define}

We will call an arc diagram \emph{clean} if it contains no trivial arcs. An arc diagram $D$ is \emph{maximal} if any nontrivial arc not already in $D$ crosses at least one arc in $D$. An arc diagram is \emph{fully homonymous} if it contains only homonymous arcs. The arcs of a diagram may also be assigned weights.

\begin{define}
A \emph{weighted arc diagram} is an arc diagram $D$ equipped with a map $w_D:D \to{} [0,\infty)$.
\end{define}

\begin{define}
A weighted arc diagram $E$ \emph{extends} a weighted arc diagram $D$ if $D \subset E$ and $w_E = w_D$ on the arcs of $D$. $E$ is also called an \emph{extension} of $D$.
\end{define}

We will also need to define a branched cover of substrates. Compare to the definitions of topological branched cover in \cite{mohar1988branched} and \cite{pikekosz1996basic}.

\begin{define}
Let $X$ and $Y$ be surfaces. A continuous map $f: X \to Y$ is called a \emph{branched covering map} if there is a finite set $\Delta \subset Y$ so that $f$ restricted to $f^{-1}(Y\backslash\Delta)$ is a covering map, and it meets the following regularity condition. For each $p \in X$ there are open neighborhoods $U_p$ of $p$ and $V_{f(p)}$ of $f(p)$ with charts $\psi:V_{f(p)} \to Z$ and $\phi:U_p \to Z$, where $Z$ is an open neighborhood of 0 in $\C$, so that the map $\psi \circ f \circ \phi^{-1}:\C \to \C$ is the complex function $z \to z^n$. Note that we are not requiring that $X$ or $Y$ have a complex structure, only that they are topological 2-manifolds.
\end{define}

The smallest such $\Delta$ is called the \emph{singular set}, and its elements are \emph{singular values}. $Y\backslash\Delta$ is called the \emph{regular set} and its elements likewise called \emph{regular values}. For each $p \in X$, the number $n$ so that $f$ is $z \to z^n$ in local coordinates around $p$ and $f(p)$ is called the \emph{local degree} of $f$ at $p$, $\deg f(p)$. A point $p$ where $\deg f(p) > 1$ is called a \emph{branch point}. 

A branched covering map is called \emph{simple} if the local degree of $f$ is 2 at each branch point and the fiber over any singular value contains a unique branch point.

\begin{define}
Let $\Sigma$ and $\Sigma^{\prime}$ be substrates. A \emph{branched cover} of $\Sigma^\prime$ by $\Sigma$ is a branched covering map $f:\Sigma \to \Sigma^\prime$ which takes $V_\Sigma$ to $V_{\Sigma^\prime}$, whose set of singular values is $M_{\Sigma^\prime}$ and which commutes with the 2-coloring. That is, for all vertices $v$ of $\Sigma$, $C_\Sigma(v)=C_{\Sigma^\prime}(f(v))$. We say that $\Sigma$ is a branched cover of $\Sigma^\prime$ if there exists such a branched covering map. 
\end{define}
We are particularly interested in branched covers of bigons. 
\begin{define}
A topological disk equipped with the structure of a substrate is a \emph{polygon}. A polygon with 2 vertices is a \emph{bigon}.
\end{define}

Arcs which share an endpoint are naturally ordered, a property we will use shortly. Suppose $D$ is an arc diagram on substrate $\Sigma$. Orient each boundary component of $\Sigma$ so that the interior of $\Sigma$ is on the left. Choose a set of smooth, disjoint, simple paths realizing $D$. The arcs incident on a given vertex may now be ordered left to right by the angle of their inward-pointing tangent vectors at $v$. We call this the \emph{canonical order} at each vertex. 

\begin{prop}\label{branching arcs exist}
Suppose $f: \Sigma \to \B$ is a simple branched cover. Let $s \in M_\B$ be a singular value, and $\wt{s}$ the branch point in the fiber over $s$. If $p$ is a simple path connecting $s$ to $v$, the union of the two lifts of $p$ which pass through $\wt{s}$ are a path representing a nontrivial homonymous arc. No other union of lifts of $p$ represents an arc.
\end{prop}
\begin{proof}
The preimage of $p$ is a collection of paths in $\Sigma$, each of which connects a point in the fiber over $s$ to a vertex which matches the color of $v$. Each path has an endpoint on a distinct vertex of $\Sigma$. Two of these paths, $a$ and $b$, will contain $\wt{s}$ as an endpoint; the rest will have endpoints on distinct points of the fiber over $s$. Therefore, only $a$ and $b$ join together into a path between vertices of $\Sigma$; none of the rest of the paths fit together into a representative of an arc. 

Now it remains to show that $a \cup b$ represents a nontrivial arc. Since $a$ and $b$ will have endpoints on distinct vertices of $\Sigma$, $a \cup b$ is not a loop. The only trivial, homonymous arcs are loops, so $a \cup b$ represents a nontrivial arc. 
\end{proof}

Suppose $f: \Sigma \to \B$ is a branched cover of a bigon by a substrate. Choose simple, disjoint paths connecting each point of $M_\B$ to a vertex of $\B$. By Proposition \ref{branching arcs exist}, each path $p$ may be asociated with a unique, nontrivial, homonymous arc $\wt{p}$ on $\Sigma$, which we call a \emph{branching arc}. We call the set of these arcs a \emph{branching diagram} for $f$. If we change these paths by homotopy rel endpoints and through paths disjoint from $M_\B$, the branching diagram remains constant. 

\begin{define}
Let $D$ be a branching diagram on $\Sigma$ for the branched cover $f$. A component of $\Sigma \backslash D$ is called a \emph{sheet}. 
\end{define}

\begin{prop}\label{sheets are simply connected}
Every sheet $s$ of a simple branched cover $f:\Sigma \to \B$ is simply connected, and $f$ is a homeomorphism on the interior of $s$. 
\end{prop}
\begin{proof}
The restriction of $f$ to the interior of $s$ is a covering map which sends interior of $(s)$ to the open set $\inside(\B)$ with the images of the branching arcs in $\bdry s$ deleted. This is a simply connected open set. Since $f$ is injective on the fundamental group of $\inside(s)$, $s$ must be simply connected.

Since $s$ is simply connected, $\chi(s)=1$. This implies that, since $f|_{\inside(s)}$ is a covering map, $1=\chi(s)=d \chi(f(s))=d$, where $d$ is the degree of $f|_{\inside(s)}$. Hence, $d=1$ and $f|_{\inside(s)}$ is a homeomorphism.
\end{proof}

An important consequence of Proposition \ref{sheets are simply connected} is that each sheet of a simple branched cover contains exactly two boundary edges of $\Sigma$.

Conversely, one can use a branching diagram together with an additional choice of order to specify a simple branched cover of the bigon. Call an arc diagram $D$ \emph{ordered} if there is a total order on the arcs of $D$ which agrees with the canonical order at each vertex. 

\begin{prop}\label{branch data}
Let $\Sigma$ be a substrate, and $D$ be an ordered arc diagram on $\Sigma$ containing $n=|V_\Sigma|/2-\chi(\Sigma)$ homonymous arcs so that each component of $\Sigma \backslash D$ is simply connected and contains two of $\Sigma$'s boundary edges. Then there exists a simple branched cover $f:\Sigma \to \B$ which has $D$ as its branching diagram.
\end{prop}
\begin{proof}
Let $d_1,\ldots,d_n$ be the arcs of $D$, and number the points of $M_\B$ $m_1,\ldots,m_n$. For $1 \leq i \leq n$, connect $m_i$ to the vertex which matches the color of the endpoints of $d_i$ by a simple path $p_i$ so that the canonical order at each vertex agrees with the order of the $p_i$'s. Let $\phi:\bdry \Sigma \cup D \to \B$ which sends $\bdry \Sigma \to \bdry \B$, sends vertices to vertices of the same color, and which sends $d_i$ to $p_i$.

For each sheet $s$, choose a homeomorphism $\phi_s:\inside(s) \to \B \backslash \phi(\bdry s)$ so that the map $f_s:s \to \B$ defined by $f_s=\phi_s$ on $\inside(s)$ and $\phi$ on $\bdry s$ is continuous. Define $f$ to be $f_s$ on each sheet $s$. 
\end{proof}

An arc on the bigon will have the same lifts under any branched cover with the same ordered branching diagram, so we will consider branched covers with the same ordered branching diagram equivalent. 

\begin{remark}
Here is an equivalent definition of substrates and arcs, which is sometimes convenient. Instead of 2-coloring the vertices of the substrate $\Sigma$, we can instead 2-color its boundary edges. An arc is now the homotopy class of a simple path disjoint from $M_\Sigma$ with each endpoint on a boundary edge, through paths disjoint from $V_\Sigma$ and $M_\Sigma$ which have their endpoints on $\bdry\Sigma$. That is, We allow the ends of the arc to slide along along a boundary edge, but not into a vertex. All other definitions and results about arc diagrams may be used with straightforward modifications. 
\end{remark}

\section{Lifting and Realizability}\label{realizability}

Suppose $\Sigma$ and $\B$ are substrates, and $B$ is a weighted arc diagram on $\B$. Let $f:\Sigma \to \B$ be a branched cover of substrates. Since $f$ is an ordinary covering map in the complement of $M_\B$, for any arc $a \in B$ we may choose a path $\hat{a}$ representing $a$ and lift $\hat{a}$ to a set of $\deg(f)$-many paths in $\Sigma$. The homotopy classes of these paths form a set of arcs called the \emph{lifts} of $a$. Some of these may be trivial arcs. We denote the set of nontrivial lifts of $a$ by $\lift(a)$. 

The branched cover $f:\Sigma \to \B$ together with the weighted arc diagram $B$ determines a clean, weighted arc diagram $\wt{B}$ on $\Sigma$ in the following way. The set of arcs in $\wt{B}$ is the union of nontrivial lifts
\[
	\bigcup_{b \in B} \lift(b)
\]
of arcs in $B$. For an arc $\wt{b} \in \wt{B}$, define the weight by
\begin{equation}
w_{\wt{B}}(\wt{b}) = \sum_{\{b \in B| \wt{b} \in \lift(b)\}} w_B(b)
\end{equation}
That is, the weight of an arc in $\wt{B}$ is the sum of the weights of all arcs which lift to it. It is convenient to treat arcs with weight zero as essentially trivial. So we impose the following equivalence relation.

\begin{define}
Arc diagrams $D$ and $D^\prime$ are \emph{equivalent} if they differ only on weight zero arcs. That is, if $a \in D \backslash D^\prime$, then $w_D(a)=0$, and likewise if $a \in D^\prime \backslash D$, then $w_{D^\prime}=0$.
\end{define}

\begin{define}
$\wt{B}$ is the \emph{lift of} $B$ \emph{under} $f$. We also say $B$ \emph{lifts to} $\wt{B}$.
\end{define}

\begin{define}
A \emph{lifting picture} is a quintuple $(\Sigma,C,\prec,\B,B)$ consisting of a substrate $\Sigma$, a branching diagram $C$, an order $\prec$ on $C$, a bigon $\B$, and an arc diagram $B$ on $\B$.
\end{define}

By Proposition \ref{branch data}, $C,\prec$ determines a branched cover $\Sigma \to \B$. 

\begin{define}
Let $L=(\Sigma,C,\prec,\B,B)$ be a lifting picture, and $D$ an arc diagram on $\Sigma$. $L$ \emph{realizes} $D$ if $B$ lifts under $f$ to $D$. $D$ is \emph{realizable} if there exists a lifting picture which realizes $D$. 
\end{define}

\begin{define}
The \emph{realizability problem} asks, given a weighted arc diagram $D$ on a substrate $\Sigma$, to decide whether $D$ is realizable, and to produce a lifting picture $(\Sigma,C,\prec,\B,B)$ which realizes $D$ if one exists.
\end{define}

Because there are only finitely many lifting pictures on a given substrate, there is a brute force solution the realizability problem. 

\begin{algo}\label{brute-force}
INPUT: A weighted arc diagram $D$ with weights $w_D$ on a substrate $\Sigma$. \\
OUTPUT: A lifting picture $(\Sigma,C,\prec,\B,B)$ realizing $D$ if one exists, or the message that no such lifting picture exists.\\
PROCEDURE: Choose a branching diagram $C$ compatible with $D$ and a valid order $\prec$ on $C$. Fix a bigon $\B$ with the appropriate number of interior marked points. Choose an unweighted arc diagram on $\B$. By propositions \ref{branching arcs exist} and \ref{branch data}, choosing $C$ and $\prec$ specifies a branched covering map $f:\Sigma \to \B$.

If the arcs of $B$ lift to those of $D$, then we get a linear map $L$ from vectors of weights $w_B$ on $B$ to vectors of weights on $D$. Compute the Moore-Penrose inverse $L^+$ of $L$. The equation $L w_B = w_D$ has a solution if and only if $L L^+ w_D = w_D$ \cite{james_1978}. In that case, $L^+ w_D$ is a solution to $L w_B = w_D$, so assign weights $L^+ w_D$ to the arcs of $w_B$, and return the resulting lifting picture $(\Sigma,C,\prec,\B,B)$. Otherwise, try a new combination of $C$, $\prec$, and unweighted diagram $B$ on $\B$. If all combinations are exhausted without yielding a solvable equation $L w_B = w_D$, then return the message that $D$ is not realizable. $$\Diamond$$
\end{algo}

\begin{prop}
A clean, maximal arc diagram $D$ on a substrate $\Sigma$ with $M_\Sigma=\emptyset$ defines a triangulation of $\Sigma$, and contains $N-3\chi(\Sigma)$ arcs, where $N$ is the cardinality of $V_\Sigma$ and $\chi(\Sigma)$ denotes the Euler characteristic.
\end{prop}
\begin{proof}
$D$ defines a cell decompostition $\mathcal{C}$ of $\Sigma$ as follows: the set of vertices $V$ of $\mathcal{C}$ is $V_\Sigma$, the set of edges $E$ contains the arcs of $D$ and the boundary edges of $\Sigma$, and the set of 2-cells $F$ consists of the closures of the components of $\Sigma \backslash E$. To see this is a triangulation, suppose there is a 2-cell $C$ that is not a triangle. If the boundary of $C$ contains at least 4 vertices, then any non-adjacent pair of vertices in $\bdry C$ can be connected by an arc which is disjoint from $D$ (as it lies in the interior of $C$) and not already in $D$. This contradicts maximality of $D$. If the boundary of $C$ contains only 1 or 2 vertices, then the interior of $C$ must have genus greater than zero; otherwise, the boundary of $C$ would collapse under homotopy. In that case, once again one can add an arc to the diagram, contradictiong maximality. This is illustrated in Figure \ref{nobigons}.

By Euler's formula, $\chi(\Sigma)=|V|-|E|+|F|$. $N=|V|$. Since $\mathcal{C}$ is a triangulation, $2|D|+\#\{\bdry\ \mathrm{edges}\}=3|F|$, where $|D|$ is the number of arcs in $D$. Since there are the same number of vertices as boundary edges, and $|E|=|D|+\#\{\bdry\ \mathrm{edges}\}$, the equation reduces to $\chi(\Sigma)= N-|D|-N+(2/3)|D|+(1/3)N$, and therefore $|D|=N-3\chi(\Sigma)$.
\end{proof}

\begin{figure}[ht]
    \centering
    \includegraphics[width=0.3\textwidth]{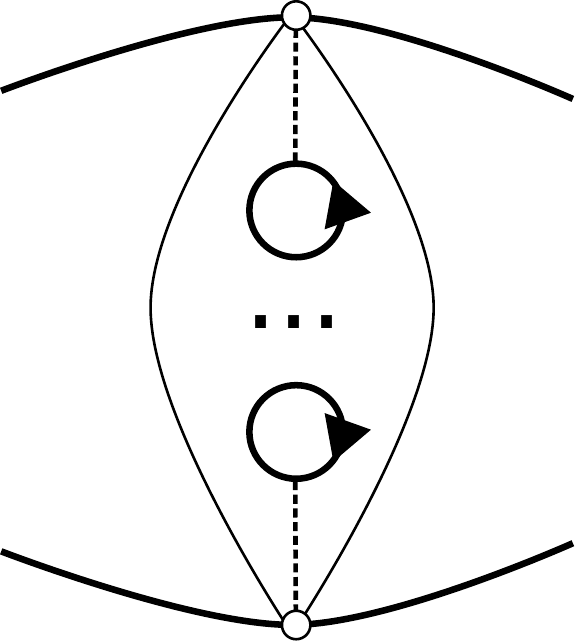}
    \caption{A subdiagram consisting of 2 distinct arcs sharing endpoints must have genus greater than 0. Bold lines represent the boundary of the surface, fine lines represent arcs. The dashed arc can be added to the diagram, meaning it is non-maximal.}
    \label{nobigons}
\end{figure}

We now prove some useful facts about lifting arcs from a bigon under a simple branched cover. 

\begin{prop}\label{how lifting works}
Let $(\Sigma,C,\prec,\B,B)$ be a lifting picture, and $b$ an arc in $B$. Let $S$ be a sheet of the cover. Then $b$ has a unique lift $\wt{b}$ in $S$. If $b$ is homonymous then $\wt{b}$ is nontrivial if and only if $b$ encloses a singular value $v$ such that the branch point $\wt{v}$ in the fiber over $v$ is contained in the boundary of $S$. If $b$ is heteronymous, then $\wt{b}$ is nontrivial if and only if $b$ separates two critical values $v$ and $w$.
\end{prop}
\begin{proof}
Since $f$ restricted to the interior of $S$ is a homeomorphism onto an open set in $\B$ which contains $b$, $b$ has a unique lift in $S$.

Suppose $b$ is homonymous. First assume $\wt{b}$ is nontrivial. Since $S$ is simply connected, either $\wt{b}$ is a branching arc in $\bdry S$ or $\wt{b}$ divides $S$ into two nonempty components. If $\wt{b}$ is a branching arc containing branch point  $u$, then $b$ is an arc which encloses $f(b)$. If $\wt{b}$ divides $S$ into two components, then one of those must contain a branch point $u$, and the other contains the boundary edges of $\Sigma$ present in $\bdry S$. Therefore, $b$ must separate $f(u)$ from $\bdry \B$ and hence $b$ encloses a singular value. Now, assume $b$ encloses a singular value $v$, whose associated branch point $\wt{v}$ is in $\bdry S$. Then since $b$ separates $v$ from $\bdry \B$, $\wt{b}$ must either separate $\wt{v}$ from $\bdry \Sigma \cap \bdry S$, or be the branching arc containing $\wt{v}$. In either case, $\wt{b}$ is nontrivial.

Now, suppose $b$ is heteronymous. First, assume $\wt{b}$ is nontrivial. Then $\wt{b}$ separates $S$ into two components. If one of them contains no branch points, and hence no branching arcs, then it must contain only $\bdry \Sigma \cap \bdry S$ and points of $\inside(S)$. But since $\bdry S$ only contains two boundary edges of $\Sigma$, then $\wt{b}$ must be homotopic to one of them, which contradicts nontriviality. So $\wt{b}$ must separate two branch points, and hence $b$ separates two singular values whose branch points are in $S$. Finally, assume $b$ separates two singular vaules whose branch points are in $S$. Then $\wt{b}$ must separate those corresponding branch points in $S$, and hence $\wt{b}$ is nontrivial.
\end{proof}

\begin{cor}
Let $(\Sigma,C,\prec,\B,B)$ be a lifting picture, and $b$ a homonymous arc in $B$ which encloses exactly one singular value $v$. Then $b$ the only nontrivial lift of $b$ is the branching arc $\wt{b}$ containing the branch point in the fiber over $b$, and $b$ lifts $\wt{b}$ twice.
\end{cor}
\begin{proof}
There are exactly two sheets whose boundary contains $\wt{b}$, and $b$ lifts to $\wt{b}$ on each. By the proposition above, $b$ will lift trivially on every other sheet.
\end{proof}

\section{Homonymous Recursion}\label{homonymous recursion section}
In this chapter, we will show how to decide whether a fully homonymous weighted arc diagram is realizable, and how to count the ways of realizing it if so. First, we will define some operations on arc diagrams and substrates which will be required in the decision algorithm. 

\begin{define}
A homonymous arc is \emph{outer} if it forms a triangle, called an \emph{outer triangle} $\delta(c)$ with two adjacent boundary edges; $c$ is said to \emph{enclose} $\delta(c)$. If it is also of minimal weight among outer arcs, it is called \emph{least outer}. If there are also arcs $a$ and $b$ in $D$ which bound a triangle with $c$ which contains no other arcs of $D$, then that triangle is unique and called the \emph{inner triangle} $\Delta(c)$. In particular, every outer arc in a maximal diagram has an inner triangle. An inner triangle is called \emph{homonymous} is it is bounded only by homonymous arcs.
\end{define}

\begin{figure}[ht]
    \centering
    \includegraphics[width=0.6\textwidth]{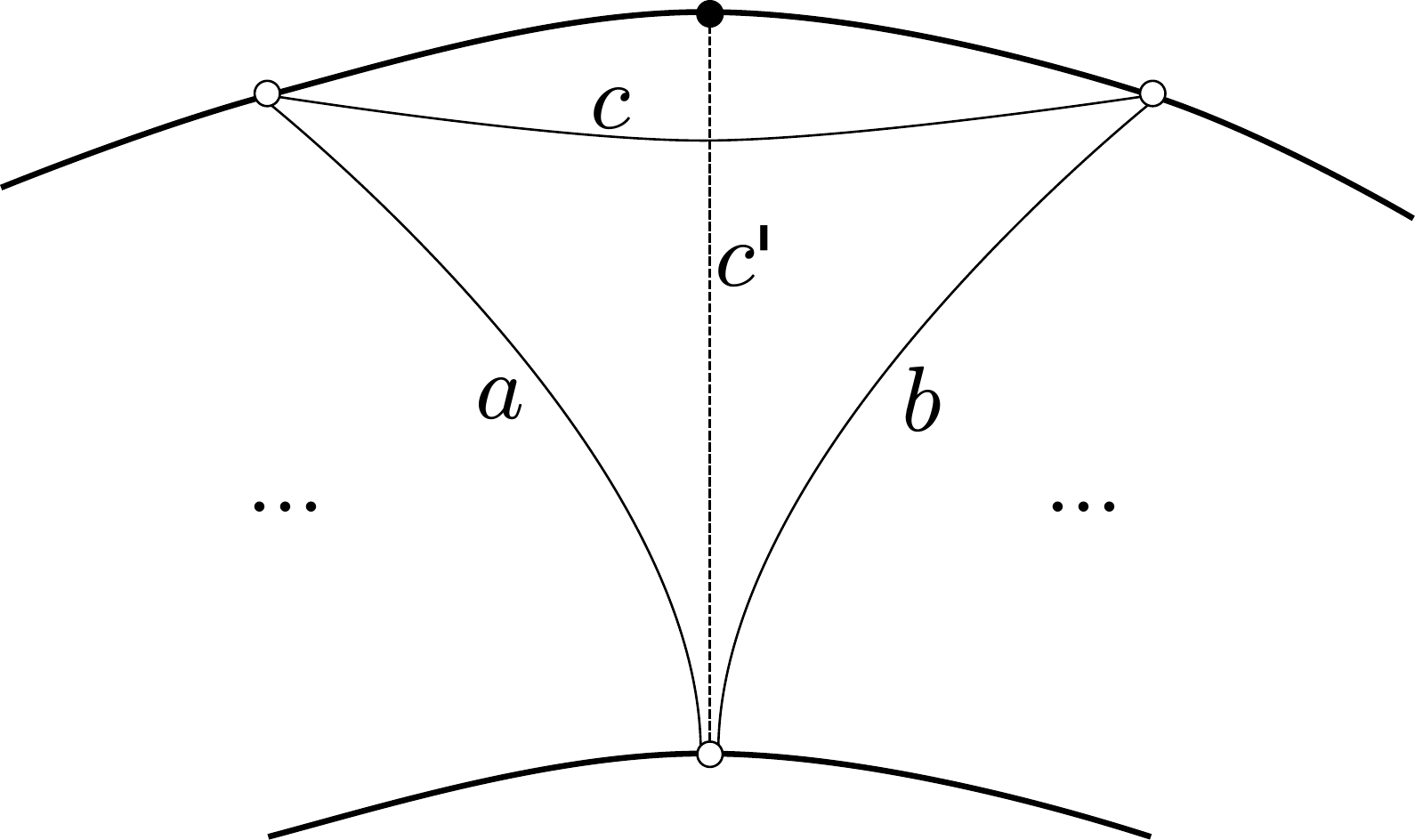}
    \caption{An outer arc $c$. $\Delta(c)$ is bounded by $a$, $b$, and $c$. The dashed arc $c^\prime$ is the dual arc of $c$.}
    \label{neartriangle}
\end{figure}
An outer arc $c$ with an inner triangle has a \emph{dual arc} $c^\prime$, which is the unique nontrivial arc crossing $c$ and connecting a vertex of $\delta(c)$ to a vertex of $\Delta(c)$. The dual arc is used to define an operation on a substrate with an arc diagram we call one seam reduction.

\begin{define}\emph{One Seam Reduction: }
Let $\Sigma$ be a substrate with an arc diagram $D$, and suppose $c$ is an outer arc with homonymous inner triangle $\Delta(c)$. Cut $\Sigma$ along the homonymous dual arc $c^\prime$; that is choose a simple representative $p^\prime$ of $c^\prime$, and let $\Sigma^\prime$ be the closure of $\Sigma \backslash c$. The \emph{one seam reduction} by $c$ is the new diagram and substrate pair $(\Sigma^\prime,D \backslash c)$.
\end{define}
\begin{figure}[ht]
    \centering
    \includegraphics[width=0.7\textwidth]{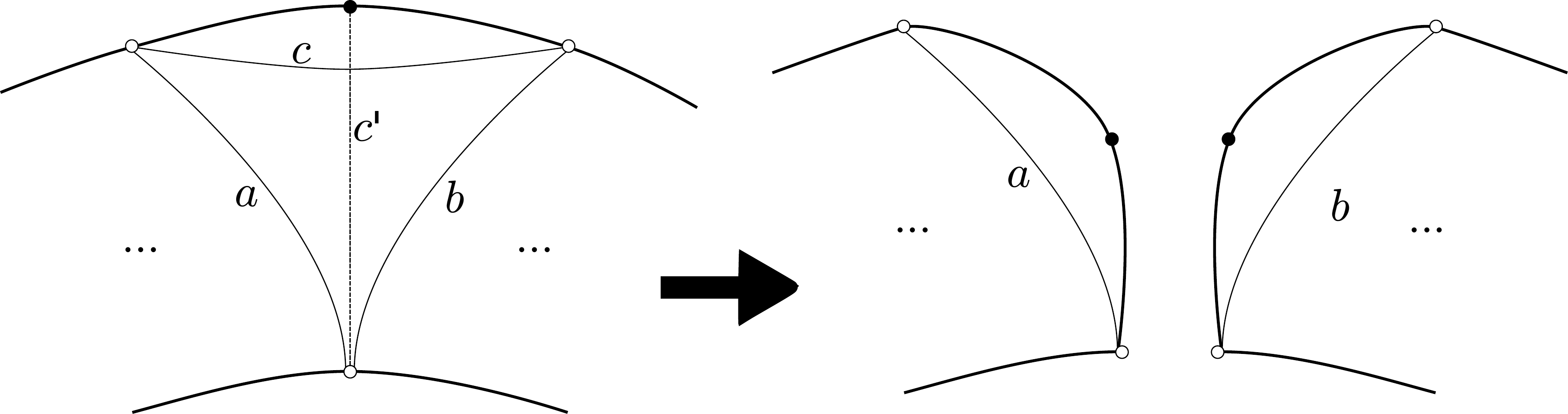}
    \caption{Before and after a one seam reduction by $c$.}
    \label{oneseamreduction}
\end{figure}
By performing a one seam reduction, you get a new arc and (possibly disconnected) substrate with a new, maximal diagram containing one fewer arc. Two components of the substrate which were separated by a one seam reduction performed on a triangle $T$ are said to \emph{meet} at $T$. Since the two legs of $T$ are both incident on a vertex, one is to the left of the other in the canonical order. If components $\Sigma^\prime$ and $\Sigma^{\prime\prime}$ meet at $T$ and $\Sigma^\prime$ contains the left leg of $T$, then $\Sigma^\prime$ meets $\Sigma^{\prime\prime}$ on the left at $T$; otherwise, it meets $\Sigma^{\prime\prime}$ on the right at $T$.

Call the reverse operation a \emph{one seam join}.

\begin{algo}\label{homonymous recursion}
\emph{Homonymous Recursion: }\\
INPUT: A substrate $\Sigma$ with a fully homonymous arc diagram $D$.\\
OUTPUT: A lifting picture $L$ realizing $D$ if one exists.\\
PROCEDURE: The base case is when $\Sigma$ is a square. $D$ has only one arc $d$, with weight $w$ and color $c$. Let $\B$ be a bigon with $M_\B$ containing a single point $p$. Let $B$ be an arc diagram consisting of a single homonymous arc $b$ with color $c$ enclosing $p$, and with weight $w/2$. The branching diagram is just $D$. Return the lifting picture $(\Sigma,D,\B,B)$.

Suppose $\Sigma$ is not a square. Let $\out(D)$ be the set of least outer arcs of $D$, which all have weight $w_\out$. Perform a one-seam reduction along each element of $\out(D)$. If the resulting substrate is connected, return that $D$ is not realizable. Order the components of the resulting substrate $\Sigma_1,\ldots,\Sigma_k$ so that if $\Sigma_i$ meets $\Sigma_j$ on the left anywhere, then $i<j$. If no such order exists, return that $D$ is not realizable. One may find this order using, for example, a depth first search. This is a special case of a topological sort, and can be done in linear time; see section 22.4 of \cite{cormen}. Perform this procedure recursively on each $\Sigma_i$. 

If none of these return that the diagram is not realizable, then they return lifting pictures $L_1,\ldots,L_k$. Let $B$ be the diagram containing diagrams $B_1,\ldots,B_k$ as subdiagrams arranged in order, with one extra arc of weight $w_\out$ enclosing all of them. The branching diagram $C$ is the union $C_1 \cup \ldots \cup C_k$ from $L_1,\ldots,L_k$. The order $\prec$ on $C$ is defined by requiring it to extend each $\prec_i$ on $C_i$; and if $c \in C_i,c^\prime \in C_j$, $i<j$, then $c \prec c^\prime$.  Return the lifting picture $(\Sigma,C,\prec,\B,B)$. $\Diamond$
\end{algo}

\begin{lemma}
Let $(\Sigma,C,\prec,\B,B)$ be a lifting picture, and suppose $b \in B$ is a homonymous arc of color $c$ which encloses every singular value. If $\Sigma$ is not a square, then the set of lifts of $b$ is the set of outer arcs of color $c$ in $\Sigma$. If $\Sigma$ is not a square, then $b$ lifts to each one once.
\end{lemma}
\begin{proof}
Each sheet $S$ contains a unique outer arc $\omega(S)$ whose outer triangle is contained in $S$. Equivalently, $\omega(S)$ separates the branching arcs in $\bdry S$ from $\bdry \Sigma \cap \bdry S$. By Proposition \ref{how lifting works}, $b$ has a unique nontrivial lift $\wt{b}$ on every sheet $S$. Since $b$ separates all singular values from $\bdry \B$, $\wt{b}$ separates the branching arcs of $\bdry S$ from $\bdry \Sigma \cap S$, so $\wt{b}=\omega(S)$.

If $\Sigma$ is not a square, then no two sheets $S,S^\prime$ can have $\omega(S)=\omega(S^\prime)$.
\end{proof}

\begin{theorem}\label{homonymous recursion works}
Let $D$ be a maximal, homonymous diagram on substrate $\Sigma$. Homonymous recursion returns a lifting picture realizing $D$ if and only if $D$ is realizable.
\end{theorem}
\begin{proof}
Suppose the lifting picture $L=(\Sigma,C,\prec,\B,B)$ is the result of performing homonymous recursion on $D$. We must show that $L$ realizes $D$. We proceed by induction on the number of arcs $|D|$ in $D$. Suppose $|D|=1$. Then $\Sigma$ must be a square. It follows from Proposition \ref{how lifting works} and its corollary that $L$ as returned by the base case of homonymous recursion realizes $D$. Now, assume that for any maximal, homonymous arc diagram $D$ with $|D|<N$ for some $N$, that if homonymous recursion returns a lifting picture, then that lifting picture realizes $D$. Suppose $D$ is a maximal, homonymous diagram with $|D|=N$. 

Let $\out$ be the set of least weighted outer arcs of $D$, and let $w_\out$ be the weight of any one of these arcs. Let $D^\prime$ be $D \backslash \out$ with the weights of all outer arcs reduced by $w_\out$, and define $\Sigma_1,\ldots,\Sigma_k$ to be the components resulting from performing a one seam reduction along each arc in $\out$. We then have maximal diagrams $D_1,\ldots,D_k$ which are the restriction of $D^\prime$ to each $\Sigma_i$. Since homonymous recursion finishes when applied to $D$, that means that $k \geq 2$ and we have lifting pictures $L_i=(\Sigma_i,C_i,\prec_i,\B_i,B_i)$, $1 \leq i \leq k$. By the inductive assumption, $L_i$ realizes $D_i$, since $|D_i|<N$. 

We now must show $C=C_1 \cup \ldots \cup C_k$ is a valid branching diagram. By construction, and using the assumption that homonymous recursion successfully terminated, the order of $\Sigma_1,\ldots,\Sigma_k$ is compatible with the canonical order at each vertex. Therefore, by induction, the order on $C$ extends the canonical order at each vertex, as it must. It remains to show that each component of $\Sigma \backslash C$ is simply connected and contains two edges of $\bdry \Sigma$. Let $S$ be a component of $\Sigma \backslash C$. If $S$ is contained in some $\Sigma_i$, then it is a valid sheet by the inductive hypothesis. Otherwise, $S$ is the result of performing one seam joins on sheets in various $\Sigma_i$'s. But since a sheet of any $\Sigma_i$ contains a unique outer arc along which to perform a one seam join, $S$ must be the result of performing a one seam join on two sheets, meaning that $S$ is simply connected and has the correct boundary. So $C$ is a valid branching diagram. 

Now we must show that $L$ realizes $D$. If $b_i$ is an arc in $B_i \subset B$, then by prop \ref{how lifting works} all nontrivial lifts will be in $\Sigma_i$, since the branch points mapped to the singular values enclosed by $b_i$ are all contained in $\Sigma_i$. Therefore, $B_i$ lifts to $D_i$ in the combined lifting picture $L$. Then the lift of $B_1 \cup \ldots \cup B_k = B \backslash \omega$, where $\omega$ is the unique arc in $B$ which encloses all of $M_\B$, is $D_1 \cup \ldots \cup D_k = D^\prime$. Since $\omega$ lifts to every outer arc once by the lemma, $L$ realizes $D$.

Finally, suppose $D$ is a realizable, maximal, homonymous weighted arc diagram on $\Sigma$. We need to show that homonymous recursion returns a lifting picture when applied to $D$. Suppose $R=(\Sigma,C_R,\prec,\B,B_R)$ is a lifting picture which realizes $D$. Let $\omega$ be the arc in $B$ which encloses $M_\B$, and consider the set $A$ of nontrivial arcs in $\B$ which cross $\omega$ and which are disjoint from the rest of $B$. The set of nontrivial lifts of arcs in $A$ is exactly the set of dual arcs of the least outer arcs of $D$. Since the arcs of $A$ divide $\B$ into multiple components, and these arcs are disjoint from the singular values, the set of lifts of these arcs must divide $\Sigma$ into at least 2 components, and these components will have a consistent left-right orientation as required in the algorithm. Therefore, one seam reduction will take place, and the algorithm will recurse onto the components. Finally, we need only observe that each component $\Sigma_i$ will have fewer boundary marked points than $\Sigma$, so therefore this process must terminate, and homonymous recursion thus returns a lifting picture.
\end{proof}

\section{Heteronymous Arcs}\label{heteronymous section}
To address the question of realizability for more general arc diagrams, we will first treat another special case.

\begin{define}
A maximal arc diagram on a substrate $\Sigma$ is \emph{maximally heteronymous} if it contains exactly $|V_\Sigma|/2-\chi(\Sigma)$ homonymous arcs.
\end{define}
A maximally heteronymous diagram has only one possible choice of branching diagram. The only freedom is in choosing a total order. As in the homonymous case, each sheet imposes an ordering on the branching arcs which form its boundary. A sheet of a maximally heteronymous diagram will be of the form shown in Figure \ref{typical het sheet}.
\begin{figure}[ht]
    \centering
    \includegraphics[width=0.3\textwidth]{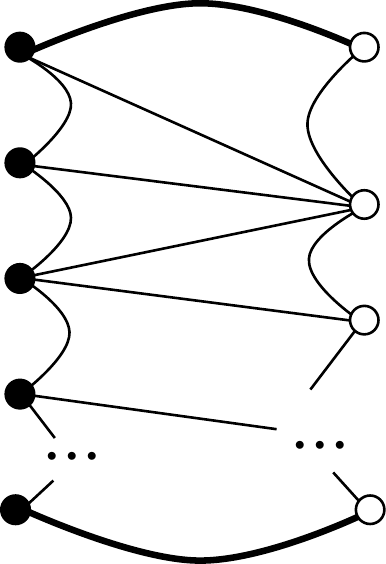}
    \caption{A typical sheet of a maximally heteronymous diagram.}
    \label{typical het sheet}
\end{figure}
For each color $c$, the canonical order at vertices orders the branching arcs of color $c$ in a sheet $S$. If $D$ is also maximal, then the heteronymous arcs contained in $S$ impose a canonical order on all branching arcs in $\bdry S$. Globally fix a color $c$. Orient each heteronymous arc so it points out of its $c$-colored endpoint. Each heteronymous arc separates $S$ into two components, and with the orientation one of these will be on the left, and the other on the right. Define an order $<_c$, on the homonymous arcs bounding $S$ by requiring $x <_c y$ if $x$ comes before $y$ in the canonical order at a common endpoint; or if there exists a heteronymous arc $h$ so that $x$ is on the left of $h$ and $y$ is on its right. Since $D$ is maximal, every pair of homonymous arcs in $S$ is separated by at least one heteronymous arc, so this is a total order on the branching arcs which bound $S$. Note that if we choose the opposite color $c^\prime$, then $x <_c y$ if and only if $y <_{c^\prime} x$. 

\begin{prop}
Let $D$ be a maximal, maximally heteronymous arc diagram on a substrate $\Sigma$. If the lifting picture $(\Sigma,C,\prec,\B,B)$ realizes $D$, then the order $\prec$ on $C$ extends $<_c$. 
\end{prop}
\begin{proof}
Assume to the contrary that $\prec$ does not extend $<_c$. Then on some sheet $S$ there is a pair of branching arcs, $x$ and $y$, and a heteronymous arc $h$, such that $x$ is left of $h$ and $y$ is on the right, but $y \prec x$. If $x$ and $y$ have the same colored endpoints, then this reverses the canonical order at a vertex, so $(\Sigma,C,\prec,\B,B)$ doesn't define a branched cover. Otherwise, since $y \prec x$, there is no arc in $B$ which separates the singular values $v_x$ from $v_y$ associated with $x$ and $y$, respectively, and puts $v_x$ to the left of $v_y$. Therefore, no arcs in $B$ can lift to $h$, and therefore $(\Sigma,C,\prec,\B,B)$ does not realize $D$.
\end{proof}

If a maximal, maximally heteronymous diagram $D$ is realizable, then it must be realized by a maximal, maximally heteronymous diagram, which on the bigon has the following form.
\begin{figure}[ht]
    \centering
    \includegraphics[width=0.5\textwidth]{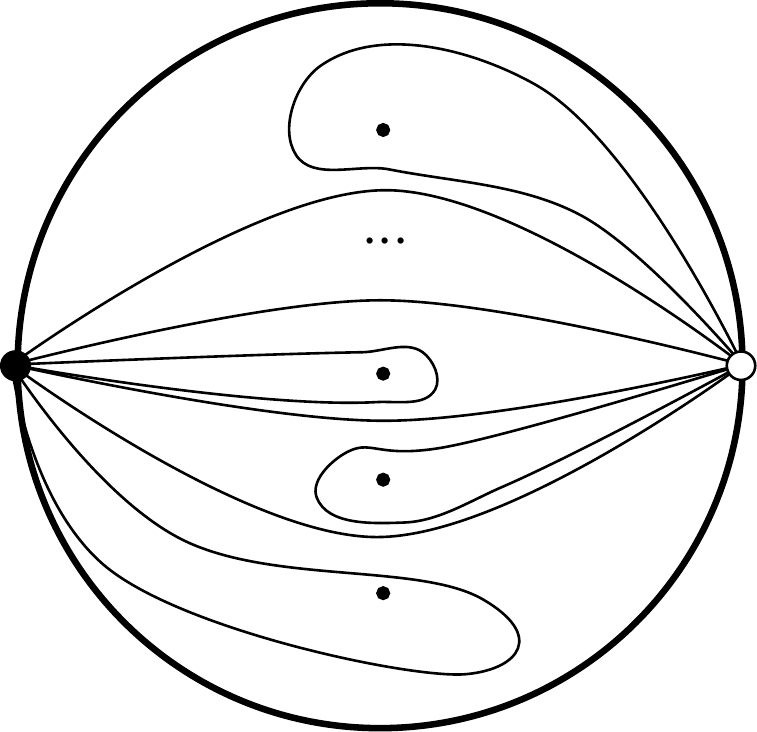}
    \caption{A maximal, maximally heteronymous diagram on the bigon. A heteronymous arc separates each pair of homonymous arcs.}
    \label{het bigon}
\end{figure}
Suppose that $a$ is a heteronymous arc in $D$, and let $f: \Sigma \to \B$ be a branched cover. $a$ is inside a sheet $S$ bounded by branching arcs, which have canonical order $<_c$. $a$ divides the branching arcs in $\bdry S$ into two sets $X$ and $Y$, with $X <_c Y$. Let $x$ be the greatest element of $X$ and $y$ be the least element of $Y$. All the hetermonymous arcs between $f(x)$ and $f(y)$ will lift to $a$. 

Therefore, to realize a maximal, maximally hetermonymous diagram $D$ is to find a total order on the homonymous arcs of $D$ which agrees with $<_c$ on each sheet and apportions weight to the heteronymous arcs of a diagram on the bigon of the form shown in Figure \ref{het bigon} so that the total weight in the bigon between any arc and its successor in a given sheet is exactly the weight of the heteronymous arc upstairs which separates them on that sheet. We will build this procedure from several more elementary algorithms.

We can represent a sheet $S$ by a list of the form 
\begin{equation}\label{list}
b_0,w_1,b_1,w_2,b_3,\ldots,w_k,b_k
\end{equation}
where the $b_i$ are the branching arcs in canonical order, and $w_i$ is the weight of the heteronymous arc between $w_{i-1}$ and $w_i$.

\begin{algo}\label{free merge}
\emph{Free Merge.}\\
INPUT: Lists $X=w_1^X,x_1,\ldots,w_m^X,x_m$ and $Y=w_1^Y,y_1,\ldots,w_n^Y,y_n$ of the form in eq. \ref{list} with $\{x_1,\ldots,x_m\} \cap \{y_1,\ldots,y_n\} = \emptyset$. \\
OUTPUT: A list $w_1^Z,z_1,\ldots,w_{m+n}^Z,z_{m+n}$.\\
PROCEDURE: The algorithm is defined recursively. The base case is when $X$ or $Y$ is empty. In that case, return the nonempty list.

If $X$ and $Y$ are both nonempty, let $w_1^Z=\min\{w_1^X,w_1^Y\}$. If $w_1^X<w_1^Y$, then let $X^\prime=w_2^X,\ldots,x_m$, $Y^\prime=w_1^Y-w_1^X,y_1,\ldots,y_n$, and $z_1=x_1$. Otherwise, let $X^\prime=w_1^X-w_1^Y,x_2,\ldots,x_m$, $Y^\prime=w_2^Y,y_2,\ldots,y_n$, and $z_1=y_1$. Return $Z=w_1^Z,z_1||FreeMerge(X^\prime,Y^\prime)$ where $||$ denotes concatenation. $\Diamond$
\end{algo}

\begin{prop}\label{free merge works}
Algorithm \ref{free merge}, when applied to lists $X=w_1^X,x_1,\ldots,w_m^X,x_m$ and $Y=w_1^Y,y_1,\ldots,w_n^Y,y_n$ returns a list $Z=w_1^Z,z_1,\ldots,w_{m+n}^Z,z_{m+n}$ satisfying the following properties:
\begin{enumerate}
\item $\{z_1,\ldots,z_{m+n}\}=\{x_1,\ldots,x_n\} \cup \{y_1,\ldots,y_m\}$
\item If $z_i=x_a$ and $z_j=x_b$, then $a<b$ implies $i<j$; likewise if $z_i=y_a$ and $z_j=y_b$, then $a<b$ implies $i<j$.
\item $w_1^X=\sum_{i=1}^k w_i^Z$, where $z_k=x_1$. Similarly, $w_1^Y=\sum_{i=1}^k w_i^Z$, where $z_k=y_1$.
\item For each $1 \leq i \leq m$, let $z_a=x_{i}$ and $z_b=x_{i+1}$. Then $w_{i+1}^X=\sum_{j=a}^b w_j^Z$, and likewise for $z_a=y_{i-1}$ and $z_b=y_{i}$.
\end{enumerate}
\end{prop}
\begin{proof}
First, observe that this algorithm always terminates after no more than $m+n$ steps of the recursion. This is because we remove the first 2 elements of either $X$ or $Y$ to make $X^\prime$ and $Y^\prime$, so at most we can only do this $m$ times to $X$ and $n$ times to $Y$. 

$Z$ trivially satisfies properties 1-4 when $X$ or $Y$ is empty. To show this in general, we will induct on the combined length of the lists, $n+m$. The base case, $n+m=1$, is a special case of $X$ or $Y$ being empty. Assume $Z$ has properties 1-4 when $n+m \leq N$ for some $N$, and suppose $n+m=N+1$. Let $X^\prime$ and $Y^\prime$ be as described in algorithm \ref{free merge}. 

We need to check that the list $Z=w_1^Z,z_1||Z^\prime$, where $Z^\prime=FreeMerge(X^\prime,Y^\prime)$, has the properties above. Properties 1 and 2 follow from the definitions of $z_1$, $X^\prime$, and $Y^\prime$, and the assumption that these properties hold for $FreeMerge(X^\prime,Y^\prime)$. To show $Z$ has property 3, WLOG assume that $w_1^Z=w_1^X$ and $z_1=x_1$; otherwise, just switch the roles of $X$ and $Y$. Let $z_k=y_1$. Then using the inductive hypothesis it suffices to show that
\[
	w_1^Y=\sum_{j=1}^k w_j^Z
\]
Now, $\sum_{j=1}^k w_j^Z=w_1^X+\sum_{j=1}^k w_j^{Z^\prime}$. By definition $w_1^{Y^\prime}=w_1^Y-w_1^X$ and by the inductive hypothesis, $w_1^{Y^\prime}=\sum_{j=1}^k w_j^{Z^\prime}$. So $\sum_{j=1}^k w_j^Z = w_1^X+w_1^{Y^\prime}=w_1^Y$. Property 4 now also follows.
\end{proof}

Now we will treat merging lists of the form in eq. \ref{list} with $x_0=y_0$ and $x_m=y_n$ so that the result also satisfies a result like Proposition \ref{free merge works}.

\begin{algo}\label{constrained merge}
\emph{Constrained Merge.}\\
INPUT: Lists $X=x_0,w_1^X,x_1,\ldots,w_m^X,x_m$ and $Y=y_0,w_1^Y,y_1,\ldots,w_n^Y,y_n$ of the form in eq. \ref{list} with $x_0=y_0$ and $x_m=y_n$, and such that
\[
	\sum_{i=1}^m w_i^X = \sum_{j=i}^n w_i^Y
\]
OUTPUT: A list $Z=z_0,w_1^Z,z_1,\ldots,w_{m+n}^Z,z_{m+n}$.\\
PROCEDURE: Let $z_0=x_0$. Compute 
\[
	Z^\prime=FreeMerge(w_1^X,x_1,\ldots,w_m^X,x_m,w_1^Y,y_1,\ldots,w_n^Y,y_n)
\]
The last 2 terms of $Z^\prime$ will be $0,x_m$ or $0,y_n$. Let $Z^{\prime\prime}$ be $Z^\prime$ with the last 2 terms deleted. Return $Z=z_0||Z^{\prime\prime}$. $\Diamond$
\end{algo}

\begin{prop}\label{constrained merge works}
Algorithm \ref{constrained merge}, when applied to lists $X=x_0,w_1^X,x_1,\ldots,w_m^X,x_m$ and $Y=y_0,w_1^Y,y_1,\ldots,w_n^Y,y_n$ satisfying the assumptions in algorithm \ref{constrained merge} returns a list $Z=z_0,w_1^Z,z_1,\ldots,w_{m+n}^Z,z_{m+n}$ satisfying the following properties:
\begin{enumerate}
\item $\{z_1,\ldots,z_{m+n}\}=\{x_1,\ldots,x_n\} \cup \{y_1,\ldots,y_m\}$
\item If $z_i=x_a$ and $z_j=x_b$, then $a<b$ implies $i<j$; likewise if $z_i=y_a$ and $z_j=y_b$, then $a<b$ implies $i<j$.
\item For each $1 \leq i \leq m+n$, let $z_a=x_{i-1}$ and $z_b=x_{i}$. Then $w_i^X=\sum_{j=a}^b w_j^Z$, and likewise for $z_a=y_{i-1}$ and $z_b=y_{i}$
\end{enumerate}
\end{prop}
\begin{proof}
Property 2 is automatic given prop \ref{free merge works}. Property 1 follows from the fact that the homonymous arcs $z_1^\prime,\ldots,z_k^\prime$ in 
\[
	Z^\prime=FreeMerge(w_1^X,x_1,\ldots,w_m^X,x_m,w_1^Y,y_1,\ldots,w_n^Y,y_n)
\]
will be the disjoint union of the arcs in the arguments of $FreeMerge$. Since the total weight in both lists is equal, the last 2 terms of $Z^\prime$ will be either $0,x_m$ or $0,y_m$. Therefore, after discarding the spurious last 2 terms of $Z^\prime$, property 1 holds, and property 3 also follows from that fact and prop \ref{free merge works}.
\end{proof}

We now generalize to all lists of the form in eq. \ref{list}. Now, we allow any number of arcs in the two lists to coincide.

\begin{algo}\label{general merge}
\emph{General Merge.}\\
INPUT: Lists $X=x_0,w_1^X,x_1,\ldots,x_m$ and $Y=y_0,w_1^Y,y_1,\ldots,y_n$ of the form in eq. \ref{list}.\\
OUTPUT: A list $Z=z_0,w_1^Z,z_1,\ldots,z_{k}$ of the form in eq. \ref{list}.\\
PROCEDURE: Compute the set of arcs $A$ which appear in both $X$ and $Y$. Define index functions $\sigma,\eta$ by, for each arc $a \in A$, $x_{\sigma(a)}=a$ and $y_{\eta(a)}=a$. For each pair $a,b$ of arcs in $A$, with $\sigma(a)<\sigma(b)$, first check that $\eta(a)<\eta(b)$. Otherwise, raise an exception. Next, check that 
\[
	\sum_{i=\sigma(a)+1}^{\sigma(b)} w_i^X = \sum_{j=\eta(a)+1}^{\eta(b)} w_j^Y
\]
If that equation is not satisfied, raise an exception. 

Now, let $A=\{a_1,\ldots,a_k\}$. Partition $X$ into sublists 
\[
	X_0=x_0,w_1^X,\ldots,x_{\sigma(a_1)-1},w_{\sigma(a_1)},
\] 
\[
	X_1=x_{\sigma(a_1)},\ldots,x_{\sigma(a_2)}
\]
\[\ldots\]
\[
	X_k=w_{\sigma(a_k)+1}^X,x_{\sigma(a_k)+1},\ldots,x_m
\] 
Partition $Y$ into sublists $Y_0,\ldots,Y_k$ in the same way. Define $Z^\prime$ to be
\[
	(FreeMerge(X_0^*,Y_0^*))^*||ConstrainedMerge(X_1,Y_1)||\ldots||FreeMerge(X_k,Y_k)
\]
where ${}^*$ denotes reindexing in opposite order. Concatenating $ConstainedMerge(X_i,Y_i)$ with $ConstainedMerge(X_{i+1},Y_{i+1})$ results in duplicate, adjacent $a_i$. Remove all these duplicates from $Z^\prime$ and call the resulting list $Z$. Return $Z$. $\Diamond$
\end{algo}

Now, as an immediate corollary of propositions \ref{free merge works} and \ref{constrained merge works}, we have

\begin{prop}\label{general merge works}
If algorithm \ref{general merge} does not raise an exception when applied to lists $X=x_0,w_1^X,x_1,\ldots,w_m^X,x_m$ and $Y=y_0,w_1^Y,y_1,\ldots,w_n^Y,y_n$ satisfying the assumptions in algorithm \ref{general merge}, then it returns a list $Z=z_0,w_1^Z,z_1,\ldots,w_{m+n}^Z,z_{m+n}$ satisfying the following properties:
\begin{enumerate}
\item $\{z_1,\ldots,z_{m+n}\}=\{x_1,\ldots,x_n\} \cup \{y_1,\ldots,y_m\}$
\item If $z_i=x_a$ and $z_j=x_b$, then $a<b$ implies $i<j$; likewise if $z_i=y_a$ and $z_j=y_b$, then $a<b$ implies $i<j$.
\item For each $1 \leq i \leq m+n$, let $z_a=x_{i-1}$ and $z_b=x_{i}$. Then $w_i^X=\sum_{j=a}^b w_j^Z$, and likewise for $z_a=y_{i-1}$ and $z_b=y_{i}$
\end{enumerate}
\end{prop}

We can now define the algorithm for realizing a maximal, maximally heteronymous arc diagram. 

\begin{algo}\label{heteronymous realization}
\emph{Heteronymous Realization.}\\
INPUT: A maximal, maximally heteronymous weighted arc diagram $D$ on a substrate $\Sigma$. \\
OUTPUT: A lifting picture $L=(\Sigma,C,\prec,\B,B)$. \\
PROCEDURE: Let $C$ be the set of homonymous arcs in $D$. Check that each component of the complement of $C$ is simply connected; if not, raise an exception. Check that $\bdry C$ contains 2 edges of $\bdry \Sigma$; otherwise, raise an exception. 

For each component $S$ of the complement of $C$, compute the list $X(S) = x_0,w_1,x_1,\ldots,w_n,x_n$, where $x_0,\ldots,x_n$ are the arcs of $C$ bounding $S$, and $w_i$ is the weight of the heteronymous arc between $x_{i-1}$ and $x_i$. Choose a component $T_0$ of $\Sigma \backslash C$. Let $Z_0=X(T_0)$. Choose a component $U$ adjacent to $T_0$. Let $T_1 = T_0 \cup U$. Compute $Z_1 = GeneralMerge(Z_0,X(U))$. If $T_1 = \Sigma$, then we are done. Otherwise, choose a component $U$ to $T_1$. Let $T_1 = T_0 \cup U$, and compute $Z_1 = GeneralMerge(Z_1,X(U))$. Continue in this fashion until, for some $k$, $T_k = \Sigma$. 

The order in which each arc of $C$ appears in $Z_k$ defines the order $\prec$ on $C$. Construct $B$ as follows. $Z_k = z_0,w_1,z_1,\ldots,w_n,z_n$. For $0 \leq i \leq n$, enclose $m_i \in M_\B$ by a homonymous arc with endpoints the same color as those of $z_i$. Assign it half the weight of $z_i$. Place a heteronymous arc between $z_i$ and $z_{i+1}$, and assign its weight to be $w_i$. Return $(\Sigma,C,\prec,\B,B)$. $\Diamond$
\end{algo}

\begin{theorem}\label{heteronymous realization works}
Let $D$ be a maximal, maximally heteronymous weighted arc diagram on a substrate $\Sigma$. Then algorithm \ref{heteronymous realization} returns a lifting picture if and only if $D$ is realizable, and the lifting picture it returns realizes $D$.
\end{theorem}
\begin{proof}
Suppose that Algorithm \ref{heteronymous realization} returns $L = (\Sigma,C,\prec,\B,B)$ when applied to $D$. It follows from Proposition \ref{how lifting works} that $L$ realizes $D$ as an unweighted diagram; it remains to show that each arc of $D$ gets the correct weight. It follows from Proposition \ref{how lifting works} that the homonymous arcs all get the correct weight. Let $h \in D$ be a heteronymous arc. It is contained in some sheet $S$. Let $c_k,c_\ell$ be the pair of arcs which flank $h$ in $S$, where the indices indicate their index in $C$ under $\prec$.  Again by Proposition \ref{how lifting works}, the heteronymous arcs in $B$ between critical values $m_k$ and $m_\ell$ will lift to $h$. The weights of those arcs are $w_{k+1},\ldots,w_\ell$. By Proposition \ref{general merge works}, $\sum_{i=k+1}^\ell w_i$ is equal to the weight of $h$, so $L$ realizes $D$.

Now, suppose $D$ is realizable. We want to show that Algorithm \ref{heteronymous realization} finishes without throwing an exception. Since $D$ is realizable, there exists a branched cover $f:\Sigma \to \B$ and a weighted arc diagram $B$ on $\B$, so that $B$ lifts to $D$ under $f$. There is a unique up to homotopy set of branch cuts disjoint from and parallel to the homonymous arcs of $B$ which connect the critical values $M_\B$ to the vertices of $\B$. Therefore, we can assume that the branch cuts lift to the homonymous arcs $C$ of $D$. That implies that there are no cycles in $C$; that each component of $D \backslash C$ is simply connected, as these are sheets of a branched cover of a disk; and that the canonical order of homonymous arcs bounding each sheet extends to a total order on $C$. Finally, we need to show that when we merge a sheet with a connected union of sheets, the condition on sums of weights in Algorithm \ref{general merge} is satisfied. This follows from Proposition \ref{how lifting works} and Lemma \ref{merge order doesn't matter}.
\end{proof}

\begin{lemma}\label{merge order doesn't matter}
Let $A,B,C$ be lists of the form required by Algorithm \ref{general merge}. Then $\M(\M(A,B),C) = \M(\M(A,C),B)$.
\end{lemma}
\begin{proof}
Let $X = \M(\M(A,B),C)$ and $Y = \M(\M(A,C),B)$, where $\M$ performs algo \ref{general merge}. Delete the arcs of $C$ not appearing in $A$ or $B$ from $Y$ and add any weights which are now adjacent; call the result $Z$. On one hand, $Z = \M(A,B)$. On the other hand, clearly $\M(Z,C)=Y$, so $X=Y$.
\end{proof}

\section{General Maximal Diagrams}\label{general maximal section}
We now put the results of the previous sections together to decide whether any maximal arc diagram is realizable. A generic maximal diagram is composed of connected, maximal, fully homonymous subdiagrams separated by heteronymous arcs. On a bigon, a generic diagram has the form shown in Figure \ref{general bigon}. 

\begin{define}
A \emph{homonymous clump} is a maximal, connected, homonymous subdiagram.
\end{define}

Let $H_1,\ldots,H_n$ be the set of homonymous clumps in $D$. Each $H_i$ has a boundary consisting of homonymous arcs which are side of a triangle with a heteronymous arc. Call these arcs \emph{outer to} $H_i$. We may cut $\Sigma$ along each of these arcs and paste in triangles as shown in Figure \ref{cutting parallel to an arc}. Let $\Sigma_i$ be the component of the resulting surface which contains $H_i$. 

\begin{figure}[ht]
    \centering
    \includegraphics[width=0.5\textwidth]{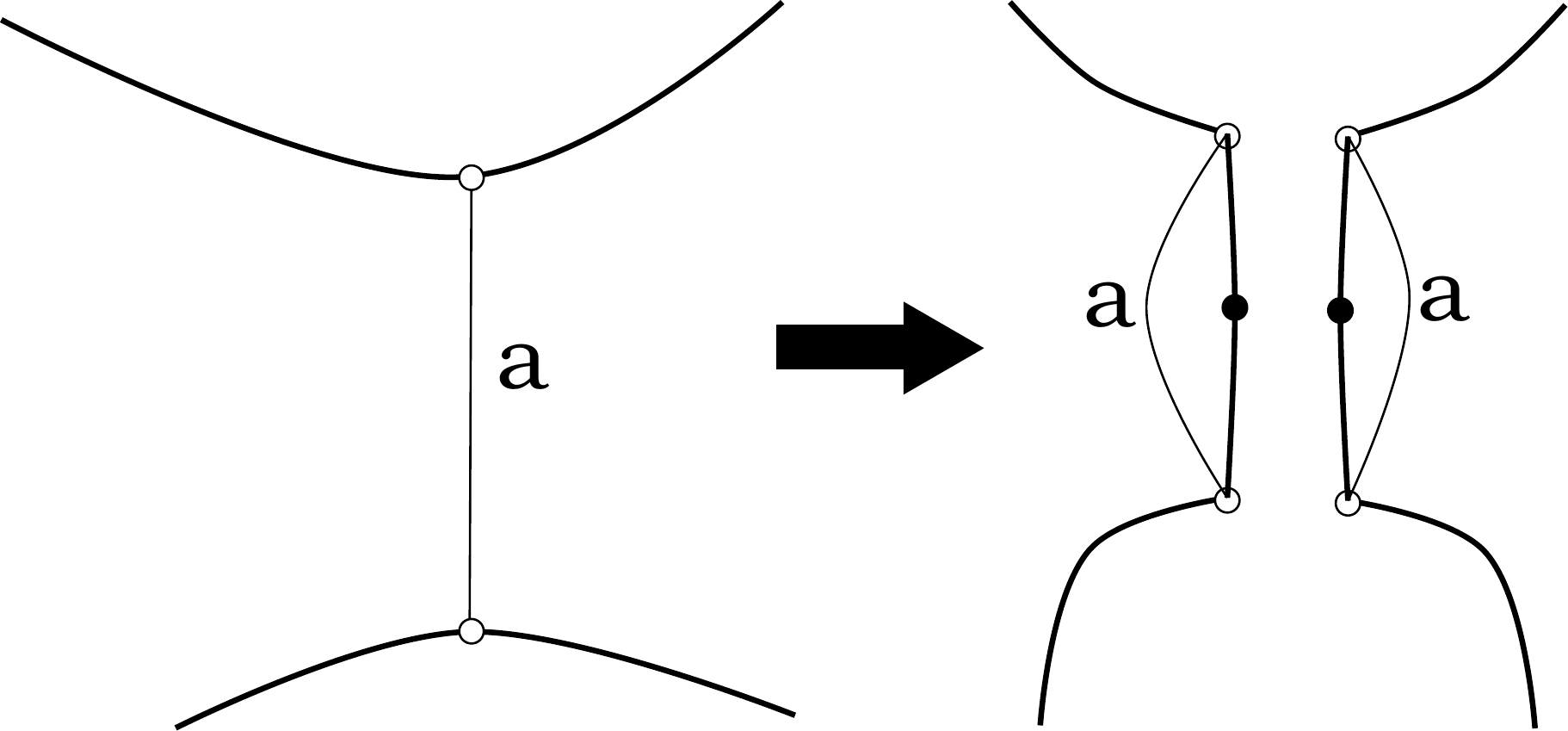}
    \caption{Cutting along the homonymous arc $a$.}
    \label{cutting parallel to an arc}
\end{figure}

\begin{figure}[ht]
    \centering
    \includegraphics[width=0.5\textwidth]{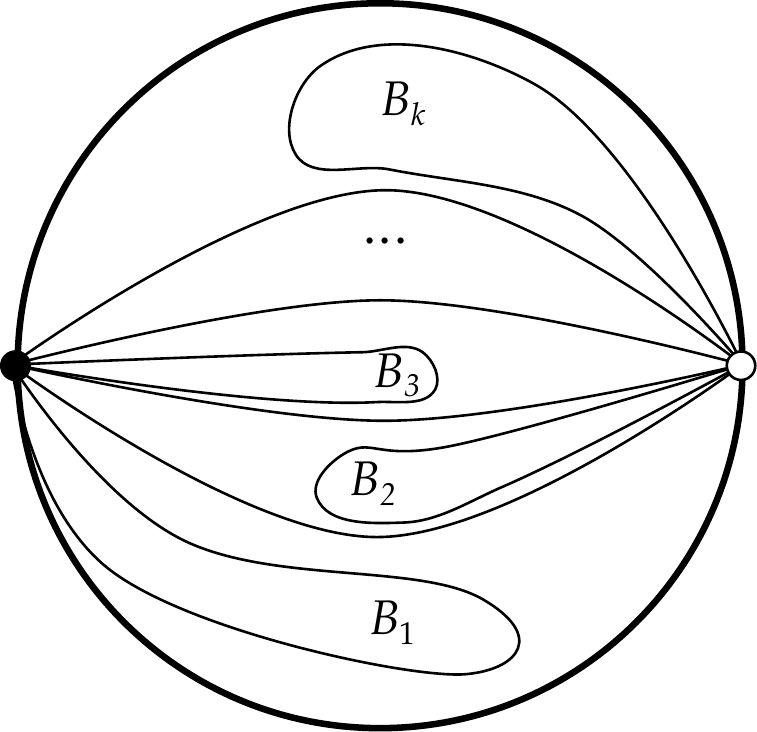}
    \caption{Each $B_i$ is a maximal, homonymous subdiagram.}
    \label{general bigon}
\end{figure}

Each component of the complement of $\Sigma_1 \cup \ldots \cup \Sigma_n$ is either an outer triangle bounded by 2 edges of $\bdry \Sigma$ and a homonymous arc, or a region of the form shown in Figure \ref{typical het sheet} which we will call a heteronymous ladder.

\begin{define}
Let $D$ be an arc diagram. A \emph{heteronymous ladder} is a region of $\Sigma$ bounded by homonymous arcs of $D$ which contains only heteronymous arcs in its interior.
\end{define}

This all suggests that to realize a diagram $D$, one can perform homonymous recursion to realize each homonymous clump individually, and then use a version of Algorithm \ref{heteronymous realization} to order these $H_i$ and assign weights to the the heteronymous arcs separating the diagrams which realize each $H_i$ downstairs. 

\begin{define}
Let $(\Sigma,C,\prec,\B,B)$ be a lifting picture realizing a diagram $D$ on $\Sigma$, and $E$ be a maximal, connected homonymous subdiagram of $B$. We say that $E$ has \emph{connected branching} if the set of branching arcs whose critical points are enclosed by arcs of $E$ all are contained in a single connected, homonymous subdiagram of $D$.
\end{define}

A maximal diagram consists of honomyous clumps and heteronymous ladders. Each heteronymous ladder can be given a canonical total order, just as in the case of the sheets of a maximally heteronymous diagram. We now show that, in essence, clumps lift to clumps. For this we'll need a procedure we call \emph{weight rebalancing}.

\begin{define}
Let $D$ be an arc diagram, $(\Sigma,C,\prec,\B,B)$ be a lifting picture realizing $D$, and $h \in B$ a homonymous arc. The \emph{branching} of $h$ is the set of arcs in $C$ whose critical points are mapped into the region enclosed by $h$. We say $h$ has \emph{connected branching} if the branching of $h$ is contained in one homonymous clump of $D$. Otherwise, $h$ has \emph{disconnected} branching. 
\end{define}

\begin{define}
Let $D$ be an arc diagram, $(\Sigma,C,\prec,\B,B)$ be a lifting picture realizing $D$, and $h \in B$ a homonymous arc. A \emph{new lift} of $h$ is an arc in $D$ which is not a lift of any arc in $B$ enclosed by $h$.
\end{define}

If a homonymous arc has no new lifts, then we may perform the following opertaion on $B$ without changing its lift. 

\begin{define}\label{pushing down weight}
\emph{Pushing down weight.}
Let $(\Sigma,C,\prec,\B,B)$ be a lifting picture, $w_B$ be the weight function on $B$, and $h \in B$ a homonymous arc. Let $h_1,\ldots,h_k$ be the outermost arcs enclosed by $h$; that is, each $h_i$ is enclosed by $h$ and not by any other arc enclosed by $h$. Define $P_h(B)$ to be $B$ with $h$ deleted and with $w_{P_h}(h_i) = w_B(h_i)+w_B(h)$. We say that we have \emph{pushed down} the weight of $h$ to make $(\Sigma,C,\prec,\B,P_h(B))$.
\end{define}

\begin{lemma}\label{pushing down weights works}
Let $D$ be a weighted arc diagram, and suppose that $(\Sigma,C,\prec,\B,B)$ realizes $D$. If $h$ is a homonymous arc in $B$ with no new lifts, then the diagram $(\Sigma,C,\prec,\B,P_h(B))$ also realizes $D$.
\end{lemma}
\begin{proof}
Let $h_1,\ldots,h_k$ be the outermost arcs enclosed by $h$. By assumption every lift of $h$ is also a lift of some $h_i$. Therefore, $(\Sigma,C,\prec,\B,P_h(B))$ still realizes $D$ as an unweighted diagram. Furthermore, since $w_{P_h}(h_i) = w_B(h_i)+w_B(h)$, the weight previously contributed by $h$ to the lifts of $h_i$ is not contributed directly by $h_i$, so $(\Sigma,C,\prec,\B,P_h(B))$ also realizes $D$ as a weighted diagram. 
\end{proof}

\begin{prop}\label{lifting pictures admit connected branching}
If $D$ is a maximal, realizable weighted arc diagram, then $D$ is realized by a lifting picture where every homonymous arc has connected branching.
\end{prop}
\begin{proof}
Suppose $(\Sigma,C,\prec,\B,B)$ realizes $D$. If there is an arc $h$ with no new lifts, push down the weight of $h$. repeat this until every arc has new lift. Call the resulting lifting picture $(\Sigma,C,\prec,\B,B^\prime)$. By Lemma \ref{pushing down weights works}, $(\Sigma,C,\prec,\B,B^\prime)$ realizes $D$.

Define $\epsilon(a)$ be the number of critical values enclosed by homonyous arc $a$. Let $h \in B^\prime$ be the arc with the smallest $\epsilon(h)$ such that $h$ has disconnected branching. We know that $h$ must have a new lift. Let $h_1,\ldots,h_k$ be the homonymous clumps enclosed by $h$. Each $h_i$ has connected branching. Since $h$ has a new lift, there must be clumps $h_a,h_b$ so that $h_a \cup h_b$ has connected branching. On the other hand, if $h_i,h_j$ are such that $h_i \cup h_j$ have disconnected branching, then we may switch their order without affecting the lift. 

Using this, we may choose a new order $h_1^\prime,\ldots,h_k^\prime$ so that if $h_i^\prime \cup h_j^\prime$ has connected branching, then there is no $h_\ell$ with $i<\ell<j$ such that $h_i \cup h_\ell \cup h_j$ has disconnected branching. That is, we can push together all the clumps under $h$ which lift to the same homonymous clump upstairs. Add a weight zero homonymous arc enclosing each maximal collection of $h_i^\prime$s whose union has connected branching. Call these arcs $\alpha_1,\ldots,\alpha_n$. Now, $h$ has no new lifts: it always lifts parallel to the lifts of $\alpha_1,\ldots,\alpha_n$. Call this diagram $B^{\prime\prime}$. Pass to $P_h(B^{\prime\prime})$. The lifting picture $(\Sigma,C,\prec,\B,P_h(B^{\prime\prime}))$ still realizes $D$. Since each of the arcs $\alpha_i$ introduced this way have connected branching, $B^{\prime\prime}$ has fewer arcs with disonnected branching than $B^\prime$. If $B^{\prime\prime}$ still has an arc with disconnected branching, choose the arc $h$ such that $h$ has the smallest $\epsilon(h)$ among arcs with disconnected branching, and eliminate it in the same way. Repeat this process until the resulting lifting picture $(\Sigma,C,\prec,\B,B^\dagger)$ has the property that all homonymous arcs of $B^\dagger$ have connected branching. Since $B^{\prime\prime}$ has only finitely many arcs, and each step of this process reduces the number of arcs with disconnected branching, this process terminates after finitely many steps. By Proposition \ref{pushing down weights works}, $B^\dagger$ still realizes $D$.

Finally, add weight zero heteronymous arcs separating each homonymous clump of $B^\dagger$. Call the resulting diagram $B^{max}$. It follows from Lemma \ref{pushing down weights works} and Lemma \ref{how lifting works} that these heteronymous arcs have no nontrivial lifts which are not also lifts of heteronymous arcs already in $B^\dagger$, so indeed $B^{max}$ still realizes $D$.
\end{proof}

Lifting pictures where all homonymous arcs have connected branching are useful, because in such lifting pictures, each homonymous clump lifts nontrivially to a unique homonymous clump.

\begin{lemma}\label{homonymous clump lemma}
Let $D$ be a maximal arc diagram on a substrate $\Sigma$. Suppose $(\Sigma,C,\prec,\B,B)$ realizes $D$, and all homonymous arcs in $B$ have connected branching. Then each homonymous clump in $B$ lifts nontrivially to a unique homonymous clump in $D$.
\end{lemma}
\begin{proof}
Let $H$ be a homonymous clump in $B$, and $h \in H$. Since $h$ has connected branching, all its nontrivial lifts are on sheets bounded by one or more branching arc in a single homonymous clump $\wt{H}$ in $D$. If $S$ is a sheet completely bounded by arcs in $\wt{H}$ then clearly any nontrivial lifts of $h$ on that sheet are arcs of $\wt{H}$. On the other hand, if $S$ is bounded by some arcs of $\wt{H}$ and some arcs from another homonymous clump, then the lift of at least one heteronymous arc in $B$ separates the lift of $h$ from the lifts of any arcs in any other homonymous clump in $B$. So $H$ lifts nontrivially only to $\wt{H}$.
\end{proof}

\begin{lemma}\label{heteronymous ladder lemma}
Let $D$ be a realizable, maximal arc diagram, and $S$ a heteronymous ladder in $D$.
\begin{enumerate}
	\item $S$ is simply connected.
	\item Each homonymous arc bounding $S$ is contained in a different homonymous clump.
\end{enumerate}
\end{lemma}
\begin{proof}
Every branching arc lies in some homonymous clump. So $S$ is contained in a single sheet of the cover, and is therefore simply connected. To prove statement 2, observe that $S$ is contained in a sheet $S^\prime$. If multiple arcs of the same homonymous clump $K$ are in $\bdry S$, then at least 2 branching arcs $a,b$ in $K$ are in $\bdry S^\prime$. Since the diagram is maximal and $S^\prime$ is simply connected, there is a heteronymous arc in $S$ which separates $a$ from $b$ in $S^\prime$. But this cannot be the lift of any arc on the bigon, as by Lemma \ref{homonymous clump lemma} no heteronymous arc on the bigon can separate two critical points from the same homonymous clump.
\end{proof}

This allows us to represent a heteronymous ladder $S$ as a list of the form needed to input into Algorithm \ref{general merge}. The homonymous arcs of $S$ can be canonically ordered just as in the maximally heteronymous case. Represent $S$ as a list $s_0,w_1,s_1,\ldots,w_n,s_n$ where $s_i$ is the homonymous clump which contains the $i^{th}$ homonymous arc bounding $S$, and $w_i$ is the weight of the heteronymous arc separating $s_{i-1}$ and $s_i$ in $S$. 

\begin{algo}\label{realization}
\emph{Maximal Diagram Realization.}\\
INPUT: A maximal, weighted arc diagram $D$ on a substrate $\Sigma$.\\
OUTPUT: A lifting picture $(\Sigma,C,\prec,\B,B)$ realizing $D$.\\
PROCEDURE:  For each homonymous clump $K$, perform homonymous recursion (that is, Algorithm \ref{homonymous recursion}) on $K$, to get a lifting picture $L_K = (\Sigma,C_K,\prec,\B_K,B_K)$. 

For each heteronymous ladder $S$, represent $S$ as a list $L(S)$. Choose a heteronymous ladder $T_0$, and let $Z_0=L(T_0)$. If this is the only heteronymous ladder, we're done. Otherwise, choose another heteronymous ladder $U$. Let $T_1 = T_0 \cup U$. Compute $Z_1 = GeneralMerge(Z_0,L(U))$. Continue in this fashion until all heteronymous ladders are merged. Let $Z=z_0,w_1^Z,z_1,\ldots,w_n,z_n$ be the resulting list.

Now construct the lifting picture. Let $\mathcal{K}$ be the set of homonymous clumps in $D$. The set $C$ of branching arcs is $\cup_{K \in \mathcal{K}} C_K$. The order $\prec$ is given by the order on each $C_K$ and the ordering on $\mathcal{K}$ provided by $Z$. The bigon diagram $B$ consists of the homonymous subdiagrams $B_{z_1},\ldots,B_{z_n}$ computed in the first part of the procedure. $B_{z_i}$ and $B_{z_{i+1}}$ are separated by a heteronymous arc with weight $w_i^Z$. Return $(\Sigma,C,\prec,\B,B)$. $\Diamond$
\end{algo}

\begin{theorem}
Let $D$ be a maximal weighted arc diagram on a substrate $\Sigma$. Then $D$ is realizable if and only if algorithm \ref{realization} finishes without raising an exception, and if it returns $L = (\Sigma,C,\prec,\B,B)$, then $L$ realizes $D$. 
\end{theorem}
\begin{proof}
Suppose Algorithm \ref{realization} returns $L = (\Sigma,C,\prec,\B,B)$. Then by Lemma \ref{homonymous clump lemma}, each maximal homonymous subdiagram $B^\prime$ of $B$ lifts to a unique homonymous clump $D^\prime$ in $D$. Since the algorithm finished without error, it follows from Theorem \ref{homonymous recursion works} that $B^\prime$ realizes $D^\prime$. It now follows from the proof of Theorem \ref{heteronymous realization works} and Lemma \ref{heteronymous ladder lemma} that the heteronymous arcs of $B$ lift to exactly the heteronymous arcs of $D$, and that they get the correct weight.

Now, suppose that $D$ is realizable. By Proposition \ref{lifting pictures admit connected branching}, we may assume $D$ is realized by a lifting picture with connected branching. Then by Lemma \ref{homonymous clump lemma}, each homonymous clump of $D$ is realizable as a homonymous diagram in its own right, so homonymous recursion will succeed on each clump. Similarly, it follows from the proof of Theorem \ref{heteronymous realization works} that $GeneralMerge$ must succeed in merging all the heteronymous ladders of $D$. So algorithm \ref{realization} finishes without error when applied to $D$.
\end{proof}

\section{Further Directions}

The motivation for the work in this dissertation came from Heegaard Floer homology. Heegaard Floer homology is a topological invariant of closed, orientable 3-manifolds, introduced by Ozsvath and Szabo in \cite{heegaard-floer-original}. It is built on top of what is known as a Heegaard diagram. Detailed information on Heegaard diagrams can be found, for example, in \cite{rolfsen}. In brief, a Heegaard splitting of a 3-manifold $M$ is a decomposition of $M$ into two handlebodies $H_1$ and $H_2$. One may reconstruct $M$ by gluing together $\bdry H_1$ and $\bdry H_2$ by a homeomorphism. It turns out that enough data to specify this gluing up to isotopy is encoded in what is called a Heegaard diagram. Let $\Sigma$ be a surface of genus $g$ homeomorphic to $\bdry H_1$ and $\bdry H_2$. A Heegaard diagram on $\Sigma$ is a choice, up to homotopy, of $2g$ closed curves on $\Sigma$, $g$ of which are labeled as coming from $H_1$ and the rest of which are labeled as coming from $H_2$. Each set of $g$ curves with the same labels must be linearly independent as 1 dimensional homology classes.

Heegaard Floer homology uses unordered $g$-tuples of intersection points of $H_1$ curves with $H_2$ curves as the generators of the Heegaard Floer chain complex. Computing the boundary map requires, among other things, counting structures known as Whitney disks. A Whitney disk, as defined in \cite{heegaard-floer-original}, can be thought of as a pair of homotopy classes of maps on a surface with boundary $\Sigma^\prime$. One map is an immersion $\phi: \Sigma^\prime \to \Sigma$ which sends $\bdry \Sigma^\prime$ to the labeled curves on $\Sigma$. The other is a branched cover $f$ taking $\Sigma^\prime$ to the unit disk with two marked points $p_1$ and $p_2$ on its boundary. For $i=1,2$, we must have that $\phi(f^{-1}(p_i))$ is an intersection point of an $H_1$ and an $H_2$ curve for each point in $f^{-1}(p_i)$. 

To compute the Heegaard Floer boundary map, one must count the number of holomorphic representatives of these Whitney disks. Holomorphic, in this case, means that for a choice of complex structure on $\Sigma$, the complex structure on $\Sigma^\prime$ one gets by pulling back the complex structure on $\Sigma$ under $\phi$ is the same as the complex structure on $\Sigma^\prime$ one gets by pulling back the standard complex structure on the unit disk under $f$. 

We ultimately wish to simplify this computation by using a more combinatorial version of a Whitney disk, which is based on a so-called degenerate complex structure. A degenerate complex structure is something like a measured lamination, which can be pulled back under immersions and branched covers, and which can be encoded, in the right circumstances, using an arc diagram. We hope that the algorithm presented here will be useful in a new method for computing Heegaard Floer homology, and that lifting pictures will be of independent interest.

\bibliography{bibliography.bib}

\begin{thebibliography}{1}

\bibitem{cormen}
Thomas H. Cormen Charles E.~Leiserson et~al.
\newblock {\em Introduction to Algorithms}.
\newblock MIT Press, 3rd edition, 2009.

\bibitem{fomin2008cluster}
Sergey Fomin, Michael Shapiro, and Dylan Thurston.
\newblock Cluster algebras and triangulated surfaces. part i: Cluster
  complexes.
\newblock {\em Acta Mathematica}, 201(1):83--146, 2008.

\bibitem{hatcher1991triangulations}
Allen Hatcher.
\newblock On triangulations of surfaces.
\newblock {\em Topology and its Applications}, 40(2):189--194, 1991.

\bibitem{james_1978}
M.~James.
\newblock The generalised inverse.
\newblock {\em The Mathematical Gazette}, 62(420):109–114, 1978.

\bibitem{korkmaz_papadopoulos_2010}
Mustafa Korkmaz and Athanase Papadopoulos.
\newblock On the arc and curve complex of a surface.
\newblock {\em Mathematical Proceedings of the Cambridge Philosophical
  Society}, 148(3):473–483, 2010.

\bibitem{mohar1988branched}
Bojan Mohar.
\newblock Branched coverings.
\newblock {\em Discrete \& Computational Geometry}, 3(4):339--348, 1988.

\bibitem{heegaard-floer-original}
Peter Ozsváth and Zoltán Szabó.
\newblock Holomorphic disks and topological invariants for closed
  three-manifolds.
\newblock {\em Annals of Mathematics}, 159(3):1027--1158, 2004.

\bibitem{pikekosz1996basic}
Artur Pi{\k{e}}kosz.
\newblock Basic definitions and properties of topological branched coverings.
\newblock {\em Topological Methods in Nonlinear Analysis}, 8(2):359--370, 1996.

\bibitem{rolfsen}
Dale Rolfsen.
\newblock {\em Knots and links}.
\newblock AMS Chelsea Publishing. American Mathematical Society, ams edition,
  2003.

\end{thebibliography}
\bibliographystyle{plain}
\end{document}